\documentclass[10pt]{amsart}
\usepackage[british,UKenglish,USenglish,american]{babel}
\usepackage[utf8]{inputenc}
\usepackage{braket}
\usepackage{xfrac}
\usepackage{amsfonts, amssymb, amsmath, amsthm, amscd}
\usepackage{mathtools}
\usepackage{tikz}
\usepackage{empheq}
\usepackage[left=3.5cm,right=3.5cm,top=3.5cm,bottom=3.5cm]{geometry}
\usepackage{color}
\usepackage{epsfig}  	
\usepackage{epic,eepic}   
\usepackage{setspace}
\usepackage{caption}
\DeclarePairedDelimiterX{\norm}[1]{\lVert}{\rVert}{#1}
\newtheorem{thm}{Theorem}[section]
\newtheorem{thmA}{Theorem}

\newtheorem{prop}[thm]{Proposition}
\newtheorem{lemma}[thm]{Lemma}

\theoremstyle{definition}

\theoremstyle{remark}
\newtheorem{remark}[thm]{Remark}

\DeclareMathOperator{\Area}{Area}

\DeclareMathOperator{\Ric}{Ric}

\DeclareMathOperator{\ind}{I}
\DeclareMathOperator{\II}{\mathrm{I \! I}}
\begin{document}


\begin{abstract}
We prove a local rigidity result for infinitesimally rigid capillary surfaces in some Riemannian $3$-manifolds with mean convex boundary. We also derive bounds on the genus, number of boundary components and area of any compact two-sided capillary minimal surface with low index under certain assumptions on the curvature of the ambient manifold and of its boundary.
\end{abstract}

\title[Low index capillary minimal surfaces in Riemannian $3$-manifolds]{Low index capillary minimal surfaces \\in Riemannian $3$-manifolds}
\author{Eduardo Longa}

\address{Departamento de Matem\'{a}tica, Instituto de Matem\'{a}tica e Estat\'{i}stica, Universidade de S\~{a}o Paulo, R. do Mat\~{a}o 1010, S\~{a}o Paulo, SP 05508-900, Brazil}
\email{eduardo.r.longa@gmail.com}



\maketitle

\let\thefootnote\relax\footnote{The author was partially supported by grant 2017/22704-0, São Paulo Research Foundation (FAPESP).}

\section{Introduction}

Minimal surfaces are critical points for the area functional under suitable constraints. If we only consider closed surfaces, no constraints are necessary. If the surfaces have a fixed boundary, this leads to the so called Plateau problem, first studied by Lagrange \cite{lagrange} and Meusnier \cite{meusnier}. However, there is a third situation that can be investigated: when we consider compact surfaces whose boundaries are allowed to move freely in the boundary of the ambient manifold. Not surprisingly, critical points of the area functional under this constraint are called \textit{free boundary} minimal surfaces. The boundary of such surfaces meet the boundary of the ambient manifold at a $90^\circ$ angle.

Although the first works dealing with this type of surfaces date back to 1938 with R. Courant (see \cite{courant1938} and \cite{courant1950}), in the last decade there have been incredible developments in this field, with the employment of new techniques and the emergence of interesting conceptual links. Among the main contributors to this topic of research we could cite Fraser, Chen and Pang, with their work on free boundary surfaces on positively curved ambients \cite{chen2014}, Ambrozio with his work on rigidity of mean-convex manifolds \cite{ambrozio2015}, and additionally Carlotto and Sharp (jointly with Ambrozio), with their works on compactness analysis and index estimates for free boundary minimal hypersurfaces (\cite{ambrozio2018} and \cite{ambrozio2018_2}).

In this paper we are interested in a natural generalisation of free boundary minimal surfaces, namely, capillary minimal surfaces. These are critical points of a certain energy functional, which will be presented in Section \ref{preliminaries_capillary}. As will be deduced later, they can be characterised as minimal surfaces whose boundary meet the ambient boundary at a constant angle --- the capillary angle. 

Capillary surfaces in $\mathbb{R}^3$ model the configuration of liquids in containers in the absence of gravity. In fact, the interface between the fluid and the air is a surface with boundary that (locally) minimises the energy functional. This energy depends on the area of the interface, the area wetted by the fluid in the container and the angle of contact between the surface and the boundary of the container. More general situations have also been considered in the literature, like the influence of gravity and the density of the fluid in the equilibrium shape. We refer the interested reader to the book of Finn \cite{finn} for an extensive survey on this subject and the derivation of the equations that describe such surfaces.

Like in the free boundary case, questions relating the topology and the geometry of surfaces raise a lot of attention from geometers. For instance, given an ambient manifold $(M^3,g)$ of a particular shape, what are the possible topological types of surfaces $\Sigma$ that admit a capillary CMC or minimal embedding into $M$? Is it possible to characterise the geometry of the allowed types? The first result in this direction was obtained by Nitsche \cite{nitsche1985}, who proved that any immersed capillary disc in the unit ball of $\mathbb{R}^3$ must be either a spherical cap or a flat disc. Later, Ros and Souam \cite{ros1997} extended this result to capillary discs in balls of $3$-dimensional space forms. Recently, Wang and Xia \cite{wang2019} analysed the problem in an arbitrary dimension and proved that any stable immersed capillary hypersurface in a ball in space forms is totally umbilical. There are many interesting uniqueness results in other types of domains, like slabs \cite{ainouz2016}, wedges \cite{park2005, choe2016}, cylinders \cite{lopez2017} and cones \cite{ritore2004}.

We are interested in the same questions raised above, but in a more general setting. Namely, we only impose curvature assumptions on the ambient $3$-manifold and look for restrictions in the topology of the possible immersed (or embedded) capillary minimal surfaces. 

When dealing with the free boundary case, Ambrozio introduced the following functional in the space of
compact and properly immersed surfaces in a Riemannian $3$-manifold $M$:
\begin{align*}
\ind(\Sigma) = \frac{1}{2} \inf_M R_M \vert \Sigma \vert + \inf_{\partial M} H_{\partial M} \vert \partial \Sigma \vert,
\end{align*}
where $R_M$ is the scalar curvature of $M$, $H_{\partial M}$ is the mean curvature of $\partial M$, $\vert \Sigma \vert$ denotes the area of $\Sigma$ and $\vert \partial \Sigma \vert$ denotes the length of $\partial \Sigma$. 

As a first result, we modify Ambrozio's functional to take care of the capillary case and we show the following theorem, generalising Proposition 6 in \cite{ambrozio2015}:

\begin{thmA} \label{inf rigidity intro}
Let $(M^3,g)$ be a Riemannian $3$-manifold with nonempty boundary, and assume that $R_M$ and $H_{\partial M}$ are bounded from below. If  $\Sigma$ is a compact two sided capillary stable minimal surface, immersed in $M$ with contact angle $\theta \in (0,\pi)$, then
\begin{align*}
\ind_\theta(\Sigma) :=  \frac{1}{2} \inf_M R_M \vert \Sigma \vert + \frac{1}{\sin \theta} \inf_{\partial M} H_{\partial M} \vert \partial \Sigma \vert \leq2 \pi \chi(\Sigma),
\end{align*}
where $\chi(\Sigma)$ denotes the Euler characteristic of $\Sigma$. Moreover, equality occurs if and only if $\Sigma$ satisfies the following properties:
\begin{itemize}
\item[(i)] $\Sigma$ is totally geodesic in $M$ and the geodesic curvature of $\partial \Sigma$ in $\partial M$ is equal to $(\cot \theta) \inf H_{\partial M}$;
\item[(ii)] the scalar curvature $R_M$ is constant along $\Sigma$ and equal to $\inf R_M$, and the mean curvature $H_{\partial M}$ is constant along $\partial \Sigma$ and equal to $\inf H_{\partial M}$;
\item[(iii)] $\Ric(N) = 0$ and $\II(\overline{\nu}, \overline{\nu}) = 0$, where $N$ is a unit normal for $\Sigma$ and $\II(\overline{\nu}, \overline{\nu})$ denotes the second fundamental form of $\partial M$ in the direction of a unit conormal for $\partial \Sigma$ in $\partial M$.
\end{itemize}
In particular, (i), (ii) and (iii) imply that $\Sigma$ has constant Gaussian curvature equal to $\inf R_M/2$ and $\partial \Sigma$ has constant geodesic curvature equal to $\inf H_{\partial M} / \sin \theta$.
\end{thmA}

A compact two-sided capillary minimal surface, properly embedded in $(M,g)$ with contact angle $\theta \in (0, \pi)$, that satisfies conditions (i), (ii) and (iii) of Theorem \ref{inf rigidity intro} will be called \textit{infinitesimally rigid}. Given one such surface $\Sigma$, there is a way to obtain a vector field $Z$ in $\Sigma$ such that $g(Z,N) = 1$ and $Z$ is tangent to $\partial M$ along $\partial \Sigma$. Let us also denote by $Z$ an extension to $M$ which is tangent to $\partial M$ along the entire boundary of  $M$. Let $\phi = \phi(x,t)$ the local flow of  $Z$ and fix a number $\alpha$ between $0$ and $1$. We show the existence of a local foliation around $\Sigma$, employing the same techniques as Ambrozio in Proposition 10 in \cite{ambrozio2015}:

\begin{thmA} \label{local foliation intro}
Let $(M^3,g)$ be a Riemannian $3$-manifold with nonempty boundary, and assume that $R_M$ and $H_{\partial M}$ are bounded from below. Let $\Sigma$ be a compact two-sided capillary minimal surface, properly embedded in $M$ with contact angle $\theta \in (0,\pi)$. If  $\Sigma$ is infinitesimally rigid, then there exists $\varepsilon > 0$ and a map $w \in C^{2,\alpha}(\Sigma \times (-\varepsilon, \varepsilon))$ such that for every $t \in (-\varepsilon, \varepsilon)$, the set
\begin{align*}
\Sigma_t = \{ \phi(x,w(x,t)) : x \in \Sigma \}
\end{align*}
is a capillary CMC surface with contact angle $\theta$ and mean curvature $H(t)$. Moreover, for each $x \in \Sigma$ and	$t \in (-\varepsilon, \varepsilon)$,
\begin{align*}
w(x,0) = 0, \quad \int_\Sigma (w(\cdot, t) - t) \, \mathrm{d} A = 0, \quad \text{and} \quad \frac{\partial w}{\partial t}(x,0) = 1.
\end{align*}
In particular, taking a smaller $\varepsilon > 0$ if necessary, $\{ \Sigma_t \}_{t \in (-\varepsilon, \varepsilon)}$ is a capillary CMC foliation of a neighbourhood of $\Sigma$ in $M$.
\end{thmA}

It is also possible to show the existence of a capillary \textit{minimal} foliation, where the contact angles now vary from leaf to leaf. Although it will not be used subsequently, we believe it may be of independent interest (see Remark \ref{variable angle}).

\begin{thmA} \label{local foliation 2 intro}
Let $(M^3,g)$ be a Riemannian $3$-manifold with nonempty boundary, and assume that $R_M$ and $H_{\partial M}$ are bounded from below. Let $\Sigma$ be a compact two-sided capillary minimal surface, properly embedded in $M$ with contact angle $\theta \in (0,\pi)$. If  $\Sigma$ is infinitesimally rigid, then there exists $\varepsilon > 0$ and a map $w \in C^{2,\alpha}(\Sigma \times (-\varepsilon, \varepsilon))$ such that for every $t \in (-\varepsilon, \varepsilon)$, the set
\begin{align*}
\Sigma_t = \{ \phi(x,w(x,t)) : x \in \Sigma \}
\end{align*}
is a capillary minimal surface with contact angle $\theta(t) \in (0, \pi)$. Moreover, for each $x \in \Sigma$ and $t \in (-\varepsilon, \varepsilon)$,
\begin{align*}
w(x,0) = 0, \quad \int_\Sigma (w(\cdot, t) - t) \, \mathrm{d} A = 0, \quad \text{and} \quad \frac{\partial w}{\partial t}(x,0) = 1.
\end{align*}
In particular, taking a smaller $\varepsilon > 0$ if necessary, $\{ \Sigma_t \}_{t \in (-\varepsilon, \varepsilon)}$ is a capillary minimal foliation of a neighbourhood of $\Sigma$ in $M$.
\end{thmA}

Then, we use Theorem \ref{local foliation intro} to show that, under some hypotheses, a dichotomy occurs: either the contact angle is equal to $\pi/2$ or a very special situation takes place. More precisely, we have:

\begin{thmA} \label{rigidity_capillary intro}
Let $(M^3,g)$ be a Riemannian $3$-manifold with nonempty and weakly mean-convex boundary, and assume that $R_M$ is bounded from below. Let $\Sigma$ be an energy-minimising and infinitesimally rigid surface, properly embedded in $M$ with contact angle $\theta \in (0,\pi)$. Assume that one of the following hypothesis holds:
\begin{itemize}
\item[(a)] each component of $\partial \Sigma$ is locally length-minimising in $\partial M$; or
\item[(b)] $\inf_{\partial M} H_{\partial M} = 0$.
\end{itemize}
Then either $\theta = \pi/2$ or $\Sigma$ is a flat and totally geodesic cylinder, $M$ is flat and $\partial M$ is totally geodesic around $\Sigma$. In the first case, there is a neighbourhood of $\Sigma$ in $M$ that is isometric to $(\Sigma \times (-\varepsilon, \varepsilon), g_\Sigma + \mathrm{d}t^2)$, where $(\Sigma, g_\Sigma)$ has constant Gaussian curvature $\inf R_M/2$ and $\partial \Sigma$ has constant geodesic curvature $\inf H_{\partial M}$ in $\Sigma$. 
\end{thmA}

A situation when $\theta \neq \pi/2$ may happen in Theorem \ref{rigidity_capillary intro} is the following. Let $P_1$ and $P_2$ be two non parallel planes in $\mathbb{R}^3$, intersecting along a line $\ell$, and let $Q$ be a plane in $\mathbb{R}^3$ which is parallel to $\ell$ and intersects both $P_1$ and $P_2$ at the same angle. Fix $S_0$ to be the (closed) wedge determined by $P_1$ and $P_2$ which contains both $Q \cap P_1$ and $Q \cap P_2$, and let $S = S_0 \setminus \ell$ (see Figure \ref{planes}). Now fix $T$ a translation of $\mathbb{R}^3$ by a vector parallel to the line $\ell$ and let $M$ be the quotient of $S$ by the group $G$ generated by $T$. If we define $\Sigma$ to be the quotient of $Q \cap S$ by $G$, then $\Sigma$ is an infinitesimally rigid cylinder intersecting $\partial M$ at a constant angle, (a) and (b) hold, $M$ is flat and $\partial M$ is totally geodesic, as we wanted. One question remains: is $\Sigma$ energy-minimising? We believe so, but we did not find a proof.

\begin{figure}[h] 
\centering
\includegraphics[width = 5cm]{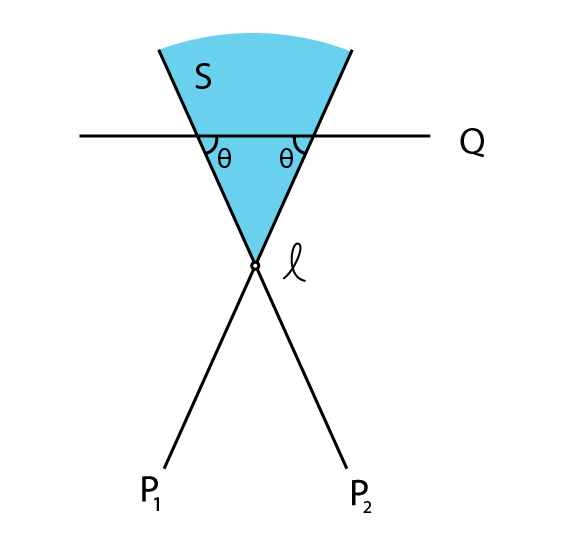}
\caption{An arrangement of $3$ planes in $\mathbb{R}^3$ seen from a plane orthogonal to the line $\ell$.}
\label{planes}
\end{figure} 

Next, we prove two other inequalities relating the geometry and the topology of capillary minimal surfaces of low index. This result generalises Theorem 1.2 in \cite{chen2014} to the capillary case. We note that item (i) below also generalises the free boundary case. 

\begin{thmA} \label{inequalities_fraser intro}
Let $(M^3,g)$ be a Riemannian $3$-manifold with nonempty boundary. Suppose that $\Sigma$ is a compact orientable two-sided capillary minimal surface of genus $g$ and with $k \geq 1$ boundary components, immersed in $M$ with contact angle $\theta \in (0, \pi)$. 
\begin{itemize}
\item[(i)] Suppose that $(M,g)$ has nonnegative Ricci curvature and weakly mean-convex boundary. If $\Sigma$ has index $1$ then
\begin{align*}
\int_{\partial \Sigma} k_g \, \mathrm{d} L < 2\pi \left[ 9 - (-1)^g - 2(g+k) \right].
\end{align*}
In particular, if the total geodesic curvature of $\partial \Sigma$ (in $\Sigma$) is nonnegative (which happens if $\Sigma$ is free boundary and $\partial M$ is weakly convex, for instance), then
\begin{itemize}
\item[(a)] $g + k \leq 3$ if $g$ is even;
\item[(b)] $g + k \leq 4$ if $g$ is odd.
\end{itemize} 
\item[(ii)] Suppose that the scalar curvature of $(M,g)$ and the mean curvature of $\partial M$ are bounded from below. If $\Sigma$ has index $1$, then
\begin{align*}
\ind_{\theta}(\Sigma) = \frac{1}{2} \inf_M R_M \vert \Sigma \vert + \frac{1}{\sin \theta} \inf_{\partial M} H_{\partial M} \vert \partial \Sigma \vert < 2 \pi\left[ 7 - (-1)^g - k \right].
\end{align*}
\item[(iii)] Suppose that $(M,g)$ has scalar curvature $R_M \geq R_0 > 0$ and weakly mean-convex boundary. 
\begin{itemize}
\item[(a)] If $\Sigma$ is stable, then it is a disc and $\vert \Sigma \vert \leq \frac{4\pi}{R_0}$.
\item[(b)] If $\Sigma$ has index $1$, then $\vert \Sigma \vert \leq \frac{4\pi \left[ 7 - (-1)^g - k \right]}{R_0}$.
\end{itemize}
\end{itemize}
\end{thmA}

\begin{remark}
The scalar curvature in \cite{chen2014} is one half of ours. This is why a factor $4 \pi$ appears in item (iii) of Theorem \ref{inequalities_fraser intro} instead of the $2 \pi$ of Theorem 1.2 in \cite{chen2014}.
\end{remark}

\section*{Acknowledgements}

The author would like to thank Paolo Piccione for his constant support and encouragement during the period when this article was written and revised, and for countless fruitful mathematical conversations. He also expresses sincere gratitude to Lucas Ambrozio for valuable comments on this work. Additionally, the author thanks Izabella Freitas and Jackeline Conrado for the figures in this paper.

\section{Preliminaries: variational problem and stability} \label{preliminaries_capillary}

The purpose of this section is to formally introduce the concept of capillary CMC and minimal hypersurfaces. 
Despite the fact that in section \ref{capillary_results} we will be dealing only with capillary surfaces in Riemannian $3$-manifolds, there is no significant simplification in introducing the main concepts only in dimension $2$. So, the general situation will be addresed in the sequel. 

Let $(M^{n+1}, g)$ be a Riemannian manifold with nonempty boundary. Let $\Sigma^n$ be a smooth compact manifold with nonempty boundary, and let $\varphi : \Sigma \to M$ be a smooth immersion of $\Sigma$ into $M$. We say that $\varphi$ is a \textit{proper immersion} if $\varphi(\Sigma) \cap \partial M = \varphi(\partial \Sigma)$.

Henceforth, we assume that $\varphi$ is two-sided. Fix a unit normal vector field $N$ for $\Sigma$ along $\varphi$ and denote by $\nu$ the outward unit conormal for $\partial \Sigma$ in $\Sigma$. Moreover, let $\overline{N}$ the outward pointing unit normal for $\partial M$ and let $\overline{\nu}$ the unit normal for $\partial \Sigma$ in $\partial M$ such that the bases $\{N, \nu\}$ and $\{\overline{N}, \overline{\nu}\}$ determine the same orientation in $(T \partial \Sigma)^\perp$. See Figure \ref{cone} to gain some intuition.

A smooth function $\Phi : \Sigma \times (-\varepsilon, \varepsilon) \to M$ is called a \textit{proper variation} of $\varphi$ is the maps $\varphi_t : \Sigma \to M$, defined by $\varphi_t(x) = \Phi(x,t)$, are proper immersions for all $t \in (-\varepsilon, \varepsilon)$, and if $\varphi_0 = \varphi$.

Let us fix a proper variation $\Phi$ of $\varphi$. The \textit{variational vector field} associated to $\Phi$ is the vector field $\xi_\Phi = \xi : \Sigma \to TM $ along $\varphi$ defined by
\begin{align*}
	\xi(x) = \frac{\partial \Phi}{\partial t}(x,0), \quad x \in \Sigma.
\end{align*}

We now define some important functionals related to the variation $\Phi$. The \textit{area functional} $A : (-\varepsilon, \varepsilon) \to \mathbb{R}$ is given by
\begin{align*}
	A(t) = \int_\Sigma \mathrm{d} A_{\varphi_t^*g},
\end{align*}
where $\mathrm{d} A_{\varphi_t^*g}$ denotes the area element of $(\Sigma, \varphi_t^\ast g)$. Even if $n = \dim \Sigma > 2$, it is customary to refer to this as the area functional.

The \textit{volume functional} $V : (-\varepsilon, \varepsilon) \to \mathbb{R}$ is defined by
\begin{align*}
	V(t) = \int_{\Sigma \times [0,t]} \Phi^*(\mathrm{d} V),
\end{align*}
where $\mathrm{d} V$ is the volume element of $M$. We say that the variation $\Phi$ is \textit{volume preserving} if $V(t) = 0$ for every $t \in (-\varepsilon, \varepsilon)$. 

We also consider the \textit{wetting area functional} $W : (-\varepsilon, \varepsilon) \to \mathbb{R}$:
\begin{align*}
	W(t) = \int_{\partial \Sigma \times [0,t]} \Phi^*(\mathrm{d} A_{\partial M}),
\end{align*}
where $\mathrm{d} A_{\partial M}$ denotes the area element of $\partial M$.

Finally, we define the  \textit{energy functional}. In order to do so,  let us fix an angle $\theta \in (0, \pi)$. Then $E_{\Phi,\theta} = E : (-\varepsilon, \varepsilon) \to \mathbb{R}$ is given by
\begin{align*}
E(t) = A(t) - (\cos \theta) W(t).
\end{align*}

The following proposition contains the formulae for the first variation of the energy and volume, whose proof can be found in \cite{spivak} and \cite{barbosa1988}:
\begin{prop} \label{first variation}
Let $\Phi$ be a proper variation of the immersion $\varphi : \Sigma \to M$. Then the following formulae hold:
\begin{align*}
E'(0) &= - \int_{\Sigma} Hf \, \mathrm{d}A + \int_{\partial \Sigma} g(\xi, \nu - (\cos \theta) \overline{\nu}) \, \mathrm{d} L \quad  \text{ and } \\
V'(0) &= \int_\Sigma f \, \mathrm{d} A, 
\end{align*}
where $f = g(\xi,N)$, $H$ is the mean curvature of $\Sigma$ with respect to $N$, $\mathrm{d} A$ is the area element  of $\Sigma$ induced by $\varphi$ and $\mathrm{d} L$ is the line element of $\partial \Sigma$ induced by $\varphi$.
\end{prop}

We say that the immersion $\varphi$ is a \textit{capillary CMC immersion} if $E'(0) = 0$ for every volume preserving variation of $\varphi$. If $E'(0) = 0$ for \textit{every} variation of $\varphi$, we call $\varphi$ a capillary minimal immersion. When there is only one immersion under consideration and there is no risk of confusion, we just say that $\Sigma$ is a capillary CMC or minimal  hypersurface. See Figure \ref{cone} for an example in an Euclidean cone.

\begin{figure}[h] 
\centering
\includegraphics[width = 4.5cm]{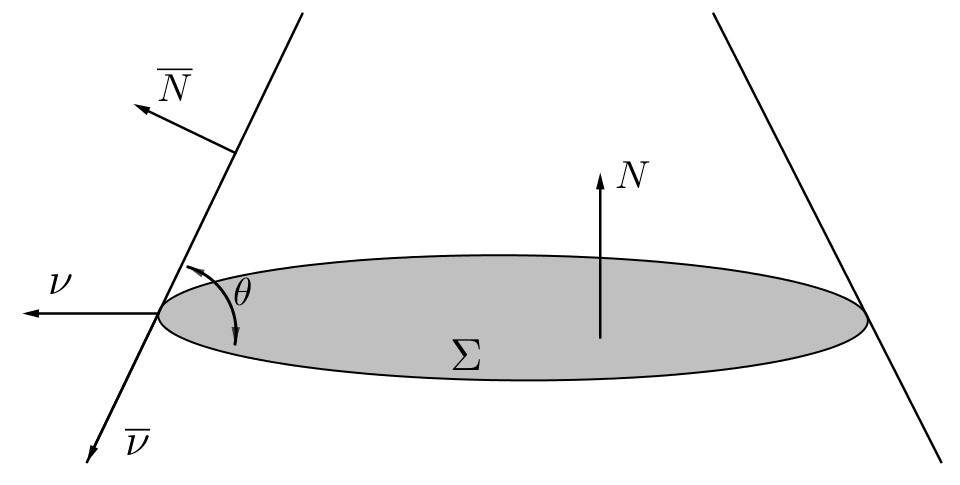}
\caption{The four fundamental vector fields in a capillary minimal surface in a cone.}
\label{cone}
\end{figure} 

Notice that $\Sigma$ is a capillary CMC hypersurface if and only if $\Sigma$ has constant mean curvature and $g(N,\overline{N}) = \cos \theta$ along $\partial \Sigma$; this last condition means that $\partial \Sigma$ meets $\partial M$ at an angle of $\theta$. Similarly, $\Sigma$ is a capillary minimal hypersurface when $\Sigma$ is a minimal hypersurface and $\partial \Sigma$ meets $\partial M$ at an angle of  $\theta$. When $\theta = \pi/2$, we use the term \textit{free boundary CMC}  (or \textit{minimal}) \textit{hypersurface}. 

For a capillary CMC or minimal hypersuface $\Sigma$ with contact angle $\theta \in (0, \pi)$, the orthonormal bases $\{N, \nu \}$ and $\{ \overline{N}, \overline{\nu} \}$ are related by the following equations:
\begin{align*}
\begin{cases}
\overline{N} = (\cos \theta) N + (\sin \theta) \nu \\
\overline{\nu} = -(\sin \theta) N + (\cos \theta) \nu
\end{cases}
\end{align*}
This will be important in the calculations.

\begin{prop}{\cite[Proposition 2.1]{ainouz2016}} \label{variation}
Let $\varphi : \Sigma^n \to M^{n+1}$ be a proper immersion which is transversal to $\partial M$. For any smooth function $f : \Sigma \to \mathbb{R}$ satisfying $\int_\Sigma f \, \mathrm{d} A = 0$, there exists a volume preserving variation $\Phi$ of $\varphi$ such that $f = g(\xi_\Phi, N)$, where $\xi = \frac{\partial \Phi}{\partial t} \vert_{t=0}$ is the variational vector field associated to $\Phi$. If we don't assume that $\int_\Sigma f \, \mathrm{d} A = 0$, then the result still holds, but now $\Phi$ doesn't need to be volume preserving.
\end{prop}

We now discuss the notion of stability. Firstly, let us fix some notation. For a proper immersion $\varphi : \Sigma \to M$, let $A(X,Y) = g(-\nabla_X N, Y)$ the second fundamental form of  $\Sigma$, and denote by $\Delta_\Sigma$ the Laplacian of $\Sigma$ with respect to the metric induced by $\varphi$. Moreover, let $\II(v,w) = g(\nabla_v \overline{N}, w)$ be the second fundamental form of  $\partial M$ with respect to $-\overline{N}$.

The following proposition gives a formula for the second variation of energy. The proof is long and can be found in the Appendix of \cite{ros1997}.

\begin{prop}
Let $\varphi : \Sigma \to M$ be a capillary CMC  immersion and let $f : \Sigma \to \mathbb{R}$ be a smooth function which satisfies $\int_\Sigma f \, \mathrm{d} A = 0$. Let $\Phi$ be a volume preserving proper variation of $\varphi$ such that $f = g(\xi_{\Phi}, N)$ (see Proposition \ref{variation}). Then,
\begin{align*}
E''(0) &= -\int_\Sigma \left[ \Delta_\Sigma f + (\Ric(N)+\norm{A}^2)f \right]f \, \mathrm{d} A 
+ \int_{\partial \Sigma} \left( \frac{\partial f}{\partial \nu} - qf \right) f \, \mathrm{d} L \\
&= - \int_\Sigma f L_\Sigma(f) \, \mathrm{d} A +  \int_{\partial \Sigma} \left( \frac{\partial f}{\partial \nu} - qf \right) f \, \mathrm{d} L,
\end{align*}
where
\begin{align*}
q = \frac{1}{\sin \theta} \II(\overline{\nu}, \overline{\nu}) + (\cot \theta) A(\nu, \nu)
\end{align*}
and $L_\Sigma = \Delta_\Sigma + (\Ric(N) + \norm{A}^2)$ is the Jacobi operator of $\Sigma$. If $\varphi$ is a capillary minimal immersion and $f$ is any smooth function defined on $\Sigma$, then the same formula holds, but now $\Phi$ doesn't need to be volume preserving.
\end{prop}

The capillary CMC immersion $\varphi : \Sigma \to M$ (or just $\Sigma$) is called \textit{stable} if $E''(0) \geq 0$ for any volume preserving variation of $\varphi$. If $\varphi$ is a capillary minimal immersion, we call it stable whenever $E''(0) \geq 0$ for every variation of $\varphi$.

Alternatively, let $\mathcal{F} = \{ f \in H^1(\Sigma) : \int_\Sigma f  \, \mathrm{d} A = 0 \}$, where $H^1(\Sigma)$ is the first Sobolev space of $\Sigma$. The index form $Q : H^1(\Sigma) \times H^1(\Sigma) \to \mathbb{R}$ of $\Sigma$ is given by
\begin{align} \label{index form}
Q(f,h) = \int_\Sigma \left[ g(\nabla f, \nabla h) - (\Ric(N)+ \norm{A}^2) f h \right] \, \mathrm{d} A - \int_{\partial \Sigma} q f h \, \mathrm{d} L.
\end{align}
Then $\varphi$ is a capillary CMC stable immersion if and only if $Q(f,f) \geq 0$ for every $f \in \mathcal{F}$. If $\varphi$ is a capillary minimal immersion, then it is stable precisely when $Q(f,f) \geq 0$ for every $f \in H^1(\Sigma)$.

The \textit{stability index} of a capillary CMC (resp. minimal) hypersurface $\Sigma$ is the dimension of the largest vector subspace of $\mathcal{F}$ (resp. $H^1(\Sigma)$) restricted to which the bilinear form $Q$ is negative definite. Thus, stable hyperfurfaces are those which have index equal to zero. Intuitively, this means that no variation of a capillary minimal stable hypersurface decreases area up to second order. 

It is possible to analitically compute the stability index of a capillary CMC (resp. minimal) hypersurface $\Sigma$ in terms of the spectrum of its Jacobi operator with a Robin boundary condition. Indeed, the boundary condition
\begin{align*}
\frac{\partial f}{\partial \nu} = q f 
\end{align*}
is elliptic for the Jacobi operator $L_\Sigma$. Therefore, there exists a nondecreasing and divergent sequence $\lambda_1 \leq \lambda_2 \leq \cdots \leq \lambda_k \nearrow \infty$ of eigenvalues for the problem
\begin{align} \label{problem}
\begin{cases}
L_\Sigma(f) + \lambda f = 0 \quad &\text{on } \Sigma \\
\frac{\partial f}{\partial \nu} = q f \quad &\text{on } \partial \Sigma
\end{cases}
\end{align}
The index of $\Sigma$ is then the number of negative eigenvalues of the problem above associated to (smooth) eigenfunctions that lie in $\mathcal{F}$ (resp. $H^1(\Sigma)$) (see more details in \cite{chen2014} and \cite{schoen2006}).

\section{Proof of the Theorems} \label{capillary_results}

This section is devoted to prove the Theorems mentioned in the Introduction. 

Recall that, given a Riemannian $3$-manifold $(M^3,g)$ with nonempty boundary and a fixed angle $\theta \in (0, \pi)$, we can define the following functional in the space of compact properly immeresd surfaces in $M$:
\begin{align*}
\ind_{\theta}(\Sigma) = \frac{1}{2} \inf_M R_M \vert \Sigma \vert + \frac{1}{\sin \theta} \inf_{\partial M} H_{\partial M} \vert \partial \Sigma \vert,
\end{align*}
provided that the scalar curvature $R_M$ and the mean curvature $H_{\partial M}$ are bounded from below.

We first show that if $\Sigma$ is capillary minimal stable, then $\ind_{\theta}(\Sigma)$ can be bounded above by the topology of $\Sigma$. We also characterise the equality case.

We begin with a simple lemma.

\begin{lemma} \label{mean curvature identity}
Let $(M^3,g)$ be a Riemannian $3$-manifold with nonempty boundary, and let $\Sigma$ be a two-sided capillary CMC surface embedded in $M$ with contact angle $\theta \in (0, \pi)$ and mean curvature equal to $H$. Then:
\begin{align*}
\II(\overline{\nu}, \overline{\nu}) + (\cos \theta) A(\nu, \nu) + (\sin \theta) k_g = H_{\partial M} + H \cos \theta
\end{align*}
\end{lemma}

\begin{proof}
Let $T$ be a unit tangent vector field along $\partial \Sigma$. Since $\Sigma$ is CMC, we know that $A(\nu, \nu) = H - A(T, T)$. So, 
\begin{align*}
(\cos \theta) A(\nu, \nu) + (\sin \theta) k_g & -H \cos \theta = - (\cos \theta) A(T,T) + (\sin \theta) k_g \\
&=g(- \nabla_T T, (\cos \theta) N + (\sin \theta) \nu) \\
&=g(- \nabla_T T, \overline{N}) \\
&=\II(T, T).
\end{align*}
But $\{T, \overline{\nu}\}$ is an orthonormal referential for $T (\partial M)$ along $\partial \Sigma$. Thus,
\begin{align*}
\II(\overline{\nu}, \overline{\nu}) + (\cos \theta) A(\nu, \nu) + (\sin \theta) k_g  = \II(\overline{\nu}, \overline{\nu}) + \II(T, T) + H \cos \theta = H_{\partial M} + H \cos \theta,
\end{align*}
as we wanted.
\end{proof}

\begin{proof}[Proof of Theorem \ref{inf rigidity intro}]
Since $\Sigma$ is capillary minimal stable, we know that $Q(f,f) \geq 0$ for any $f \in H^1(\Sigma)$, where $Q$ is the index form given by (\ref{index form}).  In particular, taking $f \equiv 1$ yields
\begin{align} \label{thmA eq1}
\int_{\Sigma} (\Ric(N) + \norm{A}^2) \, \mathrm{d} A + \int_{\partial \Sigma} q \, \mathrm{d} L \leq 0.
\end{align}
By the Gauss equation, we have:
\begin{align*} 
\Ric(N) = \frac{1}{2}(R_M + H_{\Sigma}^2 - \norm{A}^2) - K_{\Sigma},
\end{align*}
where $H_{\Sigma} \equiv 0$ and $K_\Sigma$ are the mean and Gaussian curvatures of $\Sigma$. Plugging this into (\ref{thmA eq1}) and using Gauss-Bonnet theorem for the term involving $K_\Sigma$, we obtain:
\begin{align*}
0 &\geq \frac{1}{2} \int_{\Sigma} (R_M + \norm{A}^2) \, \mathrm{d} A - \left( 2 \pi \chi(\Sigma) - \int_{\partial \Sigma} k_g \, \mathrm{d} L \right) +  \frac{1}{\sin \theta} \int_{\partial \Sigma}  \left[ \II(\overline{\nu}, \overline{\nu}) + (\cot \theta) A(\nu, \nu) \right] \, \mathrm{d} L \\
&= \frac{1}{2} \int_{\Sigma} (R_M + \norm{A}^2) \, \mathrm{d} A +  \frac{1}{\sin \theta} \int_{\partial \Sigma}  \left[ \II(\overline{\nu}, \overline{\nu}) + (\cot \theta) A(\nu, \nu) + (\sin \theta) k_g \right] \, \mathrm{d} L - 2 \pi \chi(\Sigma)  \\
&\geq \frac{1}{2} \int_{\Sigma} R_M \, \mathrm{d} A +  \frac{1}{\sin \theta} \int_{\partial \Sigma}  \left[ \II(\overline{\nu}, \overline{\nu}) + (\cot \theta) A(\nu, \nu) + (\sin \theta) k_g \right] \, \mathrm{d} L - 2 \pi \chi(\Sigma).
\end{align*}
Using Lemma \ref{mean curvature identity} with $H = 0$, we have:
\begin{align*}
0 &\geq \frac{1}{2} \int_{\Sigma} R_M \, \mathrm{d} A +  \frac{1}{\sin \theta} \int_{\partial \Sigma} H_{\partial M} \, \mathrm{d} L - 2 \pi \chi(\Sigma) \\
&\geq \frac{1}{2} \inf_M R_M \vert \Sigma \vert + \frac{1}{\sin \theta} \inf_{\partial M} H_{\partial M} \vert \partial \Sigma \vert - 2 \pi \chi(\Sigma),
\end{align*}
proving the first assertion of the theorem.
	
Let us assume now that equality holds. It is immediate to see that $\Sigma$ must be totally geodesic, item (b) holds and $Q(1,1) = 0$. The latter implies that $Q(1,f) = 0$ for any $f \in H^1(\Sigma)$. Indeed, if $c > 0$, then
\begin{align*}
0 \leq Q(1 - cf, 1 - cf) = Q(1,1) - 2 cQ(1,f) + c^2 Q(f,f) = -2c Q(1,f) + c^2 Q(f,f).
\end{align*}
Dividing by $c$ and rearranging, we obtain
\begin{align*}
Q(1,f) \leq \frac{c}{2} Q(f,f).
\end{align*}
Now we let $c \searrow 0$ and conclude that $Q(1,f) \leq 0$. The reverse inequality is analogous. This proves the claim. Thus, we know that
\begin{align*}
Q(1,f) = \int_\Sigma  \Ric(N) f  \, \mathrm{d} A + \int_{\partial \Sigma} q f \, \mathrm{d} L = 0
\end{align*}
for any $f \in H^1(\Sigma)$. This implies that $\Ric(N) = 0$ and $q = 0$. Since we already know that $A = 0$, it follows that $\II(\overline{\nu}, \overline{\nu}) = 0$, proving item (c). 
	
It remains to calculate the geodesic curvature $\overline{k}_g$ of $\Sigma$ in $\partial M$. We have
\begin{align*}
\overline{k}_g := g(-\nabla_T T, \overline{\nu}) = g(-\nabla_T T, -(\sin \theta) N + (\cos \theta) \nu) = -(\sin \theta)A(T,T) + (\cos \theta) k_g = (\cos \theta) k_g.
\end{align*}
But Lemma \ref{mean curvature identity} implies that $(\sin \theta)k_g = H_{\partial M}$. So, $\overline{k}_g = (\cot \theta) H_{\partial M}$, proving (a). The converse statement can be easily proved.
\end{proof}

We are amost ready to prove Theorem \ref{local foliation intro}. Recall that $Z$ is a vector field in $M$ which is tangent to $\partial M$ along the entire boundary of $M$ and which satisfies $g(Z, N) = 1$ along $\Sigma$. Let $\varphi = \varphi(x,t)$ be the local flow of $Z$ and fix $\alpha \in (0,1)$.

We begin with a lemma, whose proof can be found in the appendices of \cite{ambrozio2015} and \cite{li2020}.

\begin{lemma} \label{variation of mean curvature}
Let $(M^3,g)$ be a Riemannian $3$-manifold with nonempty boundary, and let $\varphi : \Sigma \to M$ a two-sided proper embedding of a compact surface $\Sigma$ with boundary. Fix a proper variation $\Phi : \Sigma \times (-\varepsilon, \varepsilon) \to M$ of $\varphi$ and let $\Sigma_t = \Phi (\Sigma \times \{t\})$. We use the subscript $t$ to denote all terms related to $\Sigma_t$, e.g. $N_t$ is the unit normal, $H_t$ is the mean curvature, etc. Let $\rho_t = g(N_t, \xi_{\Phi, t})$ and let $\theta_t$ be the contact angle between $\Sigma_t$ and $\partial M$. Then the following formulae hold:
\begin{align*} 
\frac{\partial H_t}{\partial t} &= L_{\Sigma_t}(\rho_t) - \mathrm{d}H_t(\xi_{\Phi,t}^{\top}) \quad \text{on } \Sigma_t, \\
\frac{\partial (\cos \theta_t)}{\partial \nu_t} &= - (\sin \theta_t) \frac{\partial \rho_t}{\partial \nu_t} + (\cos \theta_t) A_t(\nu_t, \nu_t) \rho_t + \II(\overline{\nu}_t, \overline{\nu}_t) \rho_t + \mathrm{d}\theta_t(W_t) \quad \text{on } \partial \Sigma_t,
\end{align*}
where $(\cdot)^\top$ denotes the orthogonal projection onto $T \Sigma_t$, $L_{\Sigma_t}$ is the Jacobi operator of $\Sigma_t$ and $W_t$ is a certain vector field along $\partial \Sigma_t$. In particular, if each $\Sigma_t$ is capillary CMC with contact angle $\theta \in (0, \pi)$, then
\begin{align}
\frac{\mathrm{d} H_t}{\mathrm{d} t} &= L_{\Sigma_t}(\rho_t)  \quad \text{on } \Sigma_t \label{variation equations 1}\\
\frac{\partial \rho_t}{\partial \nu_t} &=  \left[ \frac{1}{\sin \theta} \II(\overline{\nu}_t, \overline{\nu}_t) + (\cot \theta) A_t(\nu_t, \nu_t) \right] \rho_t \quad \text{on } \partial \Sigma_t \label{variation equations 2}
\end{align}
\end{lemma}

\begin{proof}[Proof of Theorem \ref{local foliation intro}]
Given a function $u \in C^{2,\alpha}(\Sigma)$, let $\Sigma_u = \{ \phi(x, u(x)) : x \in \Sigma \}$. If the norm of $u$ is small enough, then $\Sigma_u$ is a properly embedded surface in $M$. As in the statement of Lemma \ref{variation of mean curvature} we use the subscript $u$ to denote the terms related to $\Sigma_u$.
	
Now consider the Banach spaces 
\begin{align*}
E &= \left\lbrace u \in C^{2,\alpha}(\Sigma) : \int_\Sigma u \, \mathrm{d} A = 0 \right\rbrace, \\
F &= \left\lbrace u \in C^{0,\alpha}(\Sigma) : \int_\Sigma u \, \mathrm{d} A = 0 \right\rbrace.
\end{align*}
For small $\varepsilon > 0$ and $\delta > 0$, define the map $\Psi : B_E(0,\delta) \times (-\varepsilon, \varepsilon) \to F \times C^{1,\alpha}(\partial \Sigma)$ by
\begin{align*}
\Psi(u, t) = \left( H_{u+t} - \frac{1}{\vert \Sigma \vert} \int_{\Sigma} H_{u+t} \, \mathrm{d} A, g(N_{u+t}, \overline{N}_{u+t}) - \cos \theta \right)
\end{align*}
where $B_E(0, \delta)$ denotes the open ball of radius $\delta$ centered at the origin of $E$. We claim that $D \Psi(0,0) : E \times \mathbb{R}$ is an isomorphism when restricted to $\{0\} \times E$. To prove this, let $v \in E$. Notice that the map $\Phi : \Sigma \times (-\varepsilon, \varepsilon) \ni (x,s) \mapsto \phi(x, sv(x)) \in M$ is a proper variation of $\Sigma$ whose variational vector field is
\begin{align*}
\xi_{\Phi}(x) = \left. {\frac{\partial}{\partial s}} \right|_{s=0} \phi(x,sv(x)) = v(x) \frac{\partial \phi}{\partial s}(x,0) = v(x) Z(x), \quad x \in \Sigma.
\end{align*}
So, using the formulae from Lemma \ref{variation of mean curvature} and the fact that the Jacobi operator of $\Sigma$ is just the Laplacian, since $\Sigma$ is infinitesimally rigid, we obtain:
\begin{align} \label{derivative Psi}
D \Psi(0,0) \cdot (v, 0) &= \left. {\frac{\mathrm{d}}{\mathrm{d} s}} \right|_{s=0} \Psi(sv, 0) = 
\left( \Delta_\Sigma v - \frac{1}{\vert \Sigma \vert} \int_\Sigma \Delta_\Sigma v \, \mathrm{d} A, -(\sin \theta) \frac{\partial v}{\partial \nu}\right) \\
&=\left( \Delta_\Sigma v - \frac{1}{\vert \Sigma \vert} \int_{\partial \Sigma} \frac{\partial v}{\partial \nu} \, \mathrm{d} A, -(\sin \theta) \frac{\partial v}{\partial \nu}\right) .
\end{align} 
If this is equal to zero, then $v$ is harmonic and satisfies a Neumann condition on the boundary of $\Sigma$. This implies that $v$ is constant, and this constant must be equal to zero since $v \in E$, which proves that $D \Psi(0,0)$ is injective. The surjectivity follows from classical results for Neumann type boundary conditions for the Laplace operator.
	
The final step is to apply the implicit function theorem: for some smaller $\varepsilon > 0$, there exists a function $(-\varepsilon, \varepsilon) \ni t \mapsto u(t) \in B_E(0, \delta)$ such that $u(0) = 0$ and $\Psi(u(t),t) = (0,0)$ for any $t \in (-\varepsilon, \varepsilon)$. This means precisely that the surfaces
\begin{align*}
\Sigma_{u(t) + t} = \{ \phi(x, u(t)(x) + t) : x \in \Sigma \}
\end{align*} 
have constant mean curvature and meet $\partial M$ along their boundaries at an angle of $\theta$.
	
Now let $w : \Sigma \times (-\varepsilon, \varepsilon) \to \mathbb{R}$ be given by $w(x,t) = u(t)(x) + t$. Then $w(x,0) = 0$ and $w(\cdot, t)  - t = u(\cdot) \in B_E(0, \delta)$. Define $G : \Sigma \times (-\varepsilon, \varepsilon) \to M$ by $G(x,t) = \phi(x, w(x,t))$ and observe that this is a proper variation of $\Sigma$ whose variational vector field is 
\begin{align*}
\xi_G(x) = \frac{\partial G}{\partial t}(x,0) = \frac{\partial w}{\partial t}(x,0) \frac{\partial \phi}{\partial t}(x,0) = \frac{\partial w}{\partial t}(x,0) Z(x), \quad x \in \Sigma.
\end{align*}
Since 
\begin{align*}
0 = \Psi(u(t), t) = \left( H_{w(\cdot, t)} - \frac{1}{\vert \Sigma \vert} \int_{\Sigma} H_{w(\cdot, t)} \, \mathrm{d} A, g(N_{w(\cdot, t)}, \overline{N}_{w(\cdot, t)}) - \cos \theta \right)
\end{align*}
for all $t \in (-\varepsilon, \varepsilon)$, differentiating this equation with respect to $t$ implies that $\left. \frac{\partial w}{\partial t} \right|_{t=0}$ satisfies the Neumann problem on $\Sigma$ (see equation (\ref{derivative Psi})). Therefore, it must be constant. But
\begin{align*}
\int_\Sigma (w(x,t) - t) \, \mathrm{d} A(x) = \int_\Sigma u(t) \, \mathrm{d} A = 0, \quad t \in (-\varepsilon, \varepsilon)
\end{align*}
since $u(t) \in E$. By differentiating this equation with respect to $t$ and evaluating it at $t = 0$ we obtain that $\int_\Sigma \left. \frac{\partial w}{\partial t} \right|_{t=0} \, \mathrm{d} A = \vert \Sigma \vert$. Thus, $\left. \frac{\partial w}{\partial t} \right|_{t=0} = 1$ as we wanted to show.
	
Finally, since $G(x,0) = \phi(x,0) = x$ for every $x \in \Sigma$, $\xi_G = Z$ is transverse to $\Sigma$ and $\Sigma$ is properly embedded, by taking a smaller $\varepsilon > 0$ if necessary, we may assume that $G$ parametrises a foliation of a neighbourhood of $\Sigma$ in $M$. This concludes the proof.
\end{proof}

The proof of Theorem \ref{local foliation 2 intro} is similar to the proof of Theorem \ref{local foliation intro}. This way, we only sketch it, indicating the main differences.

\begin{proof}[Proof of Theorem \ref{local foliation 2 intro}]
As in the previous proof, for $u \in C^{2,\alpha}(\Sigma)$, let $\Sigma_u = \{ \phi(x, u(x)) : x \in \Sigma \}$.
	
Now consider the Banach spaces 
\begin{align*}
E &= \left\lbrace u \in C^{2,\alpha}(\Sigma) : \int_\Sigma u \, \mathrm{d} A = 0 \right\rbrace, \\
G &= \left\lbrace u \in C^{1,\alpha}(\partial \Sigma) : \int_{\partial \Sigma} u \, \mathrm{d} L = 0 \right\rbrace.
\end{align*}
For small $\varepsilon > 0$ and $\delta > 0$, define the map $\Lambda : B_E(0,\delta) \times (-\varepsilon, \varepsilon) \to G \times C^{0,\alpha}(\Sigma)$ by
\begin{align*}
\Lambda(u, t) = \left( g(N_{u+t}, \overline{N}_{u+t}) - \frac{1}{\vert \partial \Sigma \vert} \int_{\partial \Sigma} g(N_{u+t}, \overline{N}_{u+t}) \, \mathrm{d} L, H_{u+t} \right),
\end{align*}
where $B_E(0, \delta)$ denotes the open ball of radius $\delta$ centered at the origin of $E$. Then, using the formulae from Lemma \ref{variation of mean curvature} and the fact that the Jacobi operator of $\Sigma$ is just the Laplacian, since $\Sigma$ is infinitesimally rigid, we have:
\begin{align} 
D \Lambda(0,0) \cdot (v, 0) &= \left. {\frac{\mathrm{d}}{\mathrm{d} s}} \right|_{s=0} \Psi(sv, 0) = 
\left( -(\sin \theta) \frac{\partial v}{\partial \nu} + \frac{\sin \theta}{\vert \partial \Sigma \vert} \int_{\partial \Sigma} \frac{\partial v}{\partial \nu} \, \mathrm{d} L, \Delta_\Sigma v \right) 
\end{align} 
for any $v \in E$. It is easy to show that $D \Lambda(0,0)$ is an isomorphism when restricted to $E \times \{0\}$. Now, just apply the implicit function theorem and proceed as in the proof of Theorem \ref{local foliation intro}.
\end{proof}

Before proving Theorem \ref{rigidity_capillary intro}, let us clarify that a closed curve $\gamma$ in $\partial M$ is locally length-minimising when every nearby closed curve in $\partial M$ has length greater than or equal to the length of $\gamma$. 

Let us start with a lemma. 

\begin{lemma} \label{symmetry lemma}
Let $(M^3, g)$ be a Riemannian $3$-manifold with nonempty boundary and let $\varphi : \Sigma \to M$ be a two-sided capillary minimal proper immersion of a compact surface $\Sigma$, with contact angle equal to $\theta \in (0, \pi)$. If $\varphi$ is energy-minimising for the angle $\theta$, then, changing the sign of its unit normal, it is energy-minimising for the angle $\pi - \theta$.
\end{lemma}

\begin{proof}
Given a proper variation $\Phi : \Sigma \times (-\varepsilon, \varepsilon) \to M$ of $\varphi$, define $\tilde{\Phi} : \Sigma \times (-\varepsilon, \varepsilon) \to M$ to be the variation given by $\tilde{\Phi}(x,t) = \Phi(x,-t)$. Denote by $W_{\Phi}$ and $W_{\tilde{\Phi}}$ the wetting areas of $\Phi$ and $\tilde{\Phi}$, respectively. We claim the $W_{\tilde{\Phi}}(t) = - W_{\Phi}(-t)$ for every $t \in (-\varepsilon, \varepsilon)$. To prove this, define $f : \partial \Sigma \times (-\varepsilon, \varepsilon) \to \mathbb{R}$ by 
\begin{align*}
f = \Phi^*(\mathrm{d} A_{\partial M})(T, \mathrm{d}s),
\end{align*}
where $T$ is a unit tangent vector field along $\partial \Sigma$ and $\mathrm{d}s$ is the standard vector field in $\mathbb{R}$. Then, for $t > 0$, we have
\begin{align*}
W_{\tilde{\Phi}}(t) &= \int_{\partial \Sigma \times [0,t]} f(x, -s) \, \mathrm{d} V(x,s) 
=  \int_{\partial \Sigma} \int_{0}^t f(x,-s) \, \mathrm{d} s \, \mathrm{d} x \\
&=- \int_{\partial \Sigma} \int_{0}^{-t} f(x,u) \, \mathrm{d} u \, \mathrm{d} x \\
&=-\int_{\partial \Sigma \times [0,-t]} f(x, u) \, \mathrm{d} V(x,u) \\
&=-W_{\Phi}(-t).
\end{align*}
The computation is analogous for $t < 0$. This proves the claim. To prove the lemma, we show that $E_{\tilde{\Phi}, \pi-\theta}(t)$ is equal to $E_{\Phi, \theta}(-t)$:
\begin{align*}
E_{\tilde{\Phi}, \pi-\theta}(t) &= A(\tilde{\Phi}(\Sigma, t)) - (\cos(\pi - \theta)) W_{\tilde{\Phi}}(t) \\
&= A(\Phi(\Sigma, -t)) - (\cos \theta) W_{\Phi}(-t) \\
&=E_{\Phi, \theta}(-t).
\end{align*}
Thus, we have
\begin{align*}
E_{\tilde{\Phi}, \pi - \theta}(t) = E_{\Phi, \theta}(-t) \geq E_{\Phi, \theta}(0) = E_{\tilde{\Phi}, \pi -\theta}(0)
\end{align*}
for any variation $\Phi$ and any $t$, which implies that $\varphi$ is energy-minimising for the angle $\pi - \theta$.
\end{proof}

\begin{proof}[Proof of Theorem \ref{rigidity_capillary intro}]
Let $G$ be a parametrisation of the capillary CMC foliation $\{\Sigma_t\}_{t \in (-\varepsilon, \varepsilon)}$ given by Theorem \ref{local foliation intro}. For each $t \in (-\varepsilon, \varepsilon)$, let $\rho_t = g(\partial_ G, N_t)$, where as usual, we use the subscript $t$ for terms related to $\Sigma_t$. Then, by equations (\ref{variation equations 1}) and (\ref{variation equations 2}) in Lemma \ref{variation of mean curvature}, we have
\begin{align} 
H'(t) &= \Delta_t \rho_t + (\Ric(N_t) + \norm{A_t}^2) \rho_t  \quad \text{on } \Sigma_t ,\label{equation 1}\\
\frac{\partial \rho_t}{\partial \nu_t} &=  q_t \rho_t = \left[ \frac{1}{\sin \theta} \II(\overline{\nu}_t, \overline{\nu}_t) + (\cot \theta) A_t(\nu_t, \nu_t) \right] \rho_t \quad \text{on } \partial \Sigma_t. \label{equation 2}
\end{align}
Notice that $\rho_0 = g(Z, N) = 1$, so we can assume $\rho_t > 0$ for any $t \in  (-\varepsilon, \varepsilon)$, maybe for some smaller $\varepsilon > 0$. Using the Gauss equation, we rewrite equation (\ref{equation 1}) as
\begin{align*}
H'(t) \rho_t^{-1} = \rho_t^{-1} \Delta_t \rho_t + \frac{1}{2}(R_{M,t} + H(t)^2 + \norm{A_t}^2) - K_t,
\end{align*}
where $R_{M,t}$ denotes the scalar curvature of $M$ restricted to $\Sigma_t$ and $K_t$ denotes the Gaussian curvature of $\Sigma_t$. Now recall that $H(t)$ is constant on $\Sigma_t$, so, integration by parts gives
\begin{align*}
H'(t) \int_{\Sigma} \rho_t^{-1} \, \mathrm{d}A_t &= \int_{\Sigma} \rho_t^{-2} \norm{\nabla \rho_t}^2 \, \mathrm{d}A_t + \int_{\partial \Sigma} \rho_t^{-1} \frac{\partial \rho_t}{\partial \nu_t} \, \mathrm{d}L_t \\
&+ \frac{1}{2} \int_{\Sigma} (R_{M, t} + H(t)^2 + \norm{A_t}^2) \, \mathrm{d}A_t -\int_{\Sigma} K_t \, \mathrm{d} A_t \\
&=\int_{\Sigma} \rho_t^{-2} \norm{\nabla \rho_t}^2 \, \mathrm{d}A_t + \int_{\partial \Sigma} q_t \, \mathrm{d}L_t \\
&+ \frac{1}{2} \int_{\Sigma} (R_{M, t} + H(t)^2 + \norm{A_t}^2) \, \mathrm{d}A_t - \int_{\Sigma} K_t \, \mathrm{d}A_t \\
&\geq  \int_{\partial \Sigma} q_t \, \mathrm{d}L_t + \frac{1}{2} \int_{\Sigma} (R_{M, t} + H(t)^2 ) \, \mathrm{d}A_t - \int_{\Sigma} K_t \, \mathrm{d}A_t.
\end{align*}
We now apply Gauss-Bonnet Theorem and Lemma \ref{mean curvature identity} to obtain
\begin{align*}
H'(t) \int_{\Sigma} \rho_t^{-1} \, \mathrm{d}A_t & \geq \int_{\partial \Sigma} (q_t + k_{g,t}) \, \mathrm{d}L_t + \frac{1}{2} \int_{\Sigma} (R_{M, t} + H(t)^2) \, \mathrm{d}A_t - 2 \pi \chi(\Sigma_t) \\
&= \frac{1}{\sin \theta} \int_{\partial \Sigma} (H_{\partial M,t} + H(t) \cos \theta) \, \mathrm{d}L_t + \frac{1}{2} \int_{\Sigma} (R_{M, t} + H(t)^2) \, \mathrm{d}A_t - 2 \pi \chi(\Sigma_t)  \\
&\geq \frac{1}{2} \int_{\Sigma} R_{M, t} \, \mathrm{d}A_t + \frac{1}{\sin \theta} \int_{\partial \Sigma} H_{\partial M,t} \, \mathrm{d}L_t + H(t) \vert \partial \Sigma_t \vert \cot \theta -2 \pi \chi(\Sigma_t).
\end{align*}
Since each $\Sigma_t$ is homeomorphic to $\Sigma$, $\chi(\Sigma_t) = \chi(\Sigma)$. Next, we use the fact that $\Sigma$ is infinitesimally rigid to apply the identity $\ind_\theta(\Sigma) = 2 \pi \chi(\Sigma)$:
\begin{align*}
H'(t) \int_{\Sigma} \rho_t^{-1} \, \mathrm{d}A_t &\geq \frac{1}{2} \inf_M R_M \vert \Sigma_t \vert + \frac{1}{\sin \theta} \inf_{\partial \Sigma} H_{\partial M} \vert \partial \Sigma_t \vert + H(t) \vert \partial \Sigma_t \vert \cot \theta -2 \pi \chi(\Sigma_t) \\
&=\frac{1}{2} \inf_M R_M (\vert \Sigma_t \vert - \vert \Sigma \vert) + \frac{1}{\sin \theta} \inf_{\partial \Sigma} H_{\partial M} (\vert \partial \Sigma_t \vert - \vert \partial \Sigma \vert) + H(t) \vert \partial \Sigma_t \vert \cot \theta .
\end{align*}
By hypothesis, $\inf H_{\partial M} \geq 0$. If each component of $\partial \Sigma$ is length-minimising, then the second term in the right hand side is nonnegative, and if $\inf H_{\partial M} = 0$, then it is equal to zero. In either case, we have
\begin{align*}
H'(t) \int_{\Sigma} \rho_t^{-1} \, \mathrm{d}A &\geq \frac{1}{2} \inf_M R_M (\vert \Sigma_t \vert - \vert \Sigma \vert) + H(t) \vert \partial \Sigma_t \vert \cot \theta \\
&=\frac{1}{2} \inf_M R_M \int_0^t \frac{\mathrm{d}}{\mathrm{d}s} \vert \Sigma_s \vert \, \mathrm{d}s + H(t) \vert \partial \Sigma_t \vert \cot \theta.
\end{align*}
By the first variation of area, 
\begin{align} 
\frac{\mathrm{d}}{\mathrm{d}s} \vert \Sigma_s \vert &= -\int_{\Sigma} H(s) \rho_s \, \mathrm{d}A_s + \int_{\partial \Sigma} g(\nu_s, \partial_t G) \, \mathrm{d}L_s \nonumber \\
&= -H(s)\int_{\Sigma} \rho_s \, \mathrm{d}A_s - \cot \theta \int_{\partial \Sigma} \rho_s \, \mathrm{d}L_s. \label{var of area}
\end{align}
Thus, we obtain the following differential inequality:
\begin{align*}
H'(t) \int_{\Sigma} \rho_t^{-1} \, \mathrm{d}A \geq -\frac{1}{2} \inf_M R_M \int_0^t \left[ H(s) \int_{\Sigma}  \rho_s \, \mathrm{d}A_s + \cot \theta \int_{\partial \Sigma} \rho_s \, \mathrm{d}L_s \right] \mathrm{d}s + H(t) \vert \partial \Sigma_t \vert \cot \theta.
\end{align*}
Let us rewrite this as
\begin{align} \label{differential inequality}
H'(t) \geq \frac{R_0}{\psi(t)} \int_0^t H(s) \xi(s) \, \mathrm{d}s + \cot \theta \left[ H(t) \vert \partial \Sigma_t \vert + \frac{R_0}{\psi(t)} \int_0^t \eta(s) \, \mathrm{d}s  \right],
\end{align}
where
\begin{itemize}
\item $R_0 = -\frac{1}{2} \inf_M R_M$ 
\item $\psi(t) = \int_{\Sigma} \rho_t^{-1} \,  \mathrm{d}A_t$
\item $\xi(t) = \int_{\Sigma}  \rho_t \, \mathrm{d}A_t$
\item $\eta(t) = \int_{\partial \Sigma} \rho_t \, \mathrm{d}L_t$
\end{itemize}
	
It is time to apply Lemma \ref{symmetry lemma}. Changing the sign of the unit normal $N$ of $\Sigma$ in the begining of the proof, if necessary, we may assume that $\theta \in (0, \pi/2]$ if $R_0 \geq 0$ and that $\theta \in [\pi/2, \pi)$ if $R_0 < 0$.  In either case, inequality (\ref{differential inequality}) implies that $H'(0) \geq 0$.  If $H'(0) > 0$ then $H(t) > 0$ for all small positive $t$. The first variation formula for the energy (see Proposition \ref{first variation}) would then give that $E(t) < E(0)$:
\begin{align*}
E(t) - E(0) = \int_0^t E'(s) \, \mathrm{d}s = - \int_0^t H(s) \left( \int_{\Sigma} \rho_s \, \mathrm{d}A_s \right) \, \mathrm{d}s < 0,
\end{align*}
which is absurd, since $\Sigma$ is energy-minimising. So $H'(0) = 0$. Now let $\alpha(t)$ denote the real function in the right hand side of inequality (\ref{differential inequality}). A simple computation yields $\alpha'(0) = \frac{\eta(0)}{\psi(0)}R_0 \cot \theta$. This is positive if and only if $R_0 \neq 0$ and $\theta \neq \pi/2$. In this case, we would have $H''(0) \geq \alpha'(0) > 0$, which is also absurd by the same reason. So either $\theta = \pi/2$, and we are done, or $R_0 = 0$. 
	
Let us analyse the case $R_0 = 0$. Inequality (\ref{differential inequality}) takes the form $H'(t) \geq  (\cot \theta) \vert \partial \Sigma_t \vert  H(t)$, and since $\cot \theta \geq 0$, this implies that $H(t) \geq 0$ for $t \geq 0$ and $H(t) \leq 0$ for $t \leq 0$. By the first variation of energy again, $E(t) \leq E(0)$ for any $t \in (-\varepsilon, \varepsilon)$. Thus, $E(t) \equiv E(0)$ because $\Sigma$ is energy-minimising. So, $H(t) \equiv 0$ by the same formula. In particular, every $\Sigma_t$ is capillary minimal stable. Applying Theorem \ref{inf rigidity intro} we obtain that, for every $t$,
\begin{align*}
2 \pi \chi(\Sigma_t) \geq \ind_\theta(\Sigma_t) = \frac{1}{\sin \theta} \inf_{\partial M} H_{\partial M} \vert \partial \Sigma_t \vert \geq \frac{1}{\sin \theta} \inf_{\partial M} H_{\partial M} \vert \partial \Sigma \vert = \ind_\theta(\Sigma) = 2 \pi \chi(\Sigma) = 2 \pi \chi(\Sigma_t).
\end{align*}
This way, every inequality above is an equality, which shows that every $\Sigma_t$ is infinitesimally rigid. In particular, $\partial \Sigma_t$ has constant geodesic curvature equal to $(\cot \theta) \inf_{\partial M} H_{\partial M}$ in $\partial M$. Since we are supposing either condition (a) or (b) in the theorem, we conclude that either $\inf_{\partial M} H_{\partial M} = 0$ or each component of $\partial \Sigma$ is a geodesic in $\partial M$. Hence, either $\theta = \pi/2$ and we are done, or $\inf_{\partial M} H_{\partial M} = 0$.
	
Assuming $R_0 = 0$ and $\inf_{\partial M} H_{\partial M} = 0$ we are going to show that $M$ is flat in a neighbourhood of $\Sigma$ and that $\partial M$ is totally geodesic in this neighbourhood.
	
\textbf{Claim}: The unit normals $N_t$ to $\Sigma_t$ define a parallel vector field.
	
Firstly, notice that each map $\rho_t$ is constant because it harmonic with Neumann boundary condition by equations (\ref{equation 1}) and (\ref{equation 2}). Now, let $(x_1, x_2)$ be local coordinates on $\Sigma$ and write $E_i = \partial_{x_i} G$. Since each $\Sigma_t$ is totally geodesic, we know that $\nabla_{E_i} N_t = 0$. Moreover, since $N_t$ is of unit length, $\nabla_{\partial_t G} N_t$ is tangent to $\Sigma_t$. Thus, 
\begin{align*}
g(\nabla_{\partial_t G} N_t, E_i) &= \partial_t g(N_t, E_i) - g(N_t, \nabla_{\partial_t G} E_i) \\
&= -g(N_t, \nabla_{E_i} \partial_t G) \\
&= -\partial_{x_i} g(N_t, \partial_t G) + g(\nabla_{E_i} N_t, \partial_t G) \\
&= -\partial_{x_i} \rho_t \\
&= 0.
\end{align*} 
This proves the Claim.
	
The Gauss equation for $\Sigma_t$ shows that $\overline{K}(\partial_{x_1} G \wedge \partial_{x_2} G) = K(\partial_{x_1} G \wedge \partial_{x_2} G) = 0$, where $K$ and $\overline{K}$ denote the sectional curvatures of $\Sigma$ and of $M$. This holds since $\Sigma_t$ is flat and totally geodesic. Now, since $N_t$ is parallel, $g(R(\cdot, \cdot) N_t, \cdot) = 0$. The symmetries of the curvature tensor then imply that $R \equiv 0$, that is, $M$ is flat around $\Sigma$.
	
Finally, let us show that $\partial M$ is totally geodesic near $\Sigma$. By infinitesimal rigidity we know that $\II(\nu_t, \nu_t) \equiv 0$ for every $t \in (-\varepsilon, \varepsilon)$. If we denote by $T_t$ a unit normal for $\partial \Sigma_t$, it is straighforward to check that $\II(T_t, \nu_t) \equiv 0$ and $\II(T_t,T_t) \equiv 0$. Since $\{\nu_t, T_t \}$ is an orthonormal referential for $T (\partial M)$ around $\partial \Sigma$, the bilinearity of $\II$ implies that $\II \equiv 0$, i.e. $\partial M$ is totally geodesic in the union of all $\partial \Sigma_t$, $t \in (-\varepsilon, \varepsilon)$. 
	
The last statement of the theorem (when $\theta = \pi/2$) follows from Theorem 7 in \cite{ambrozio2015}.
\end{proof}

\begin{remark} \label{variable angle}
Theorem \ref{rigidity_capillary intro} can be proved using the capillary minimal foliation with variable contact angle given by Theorem \ref{local foliation 2 intro} instead of the capillary CMC foliation of Theorem \ref{local foliation intro}.
\end{remark}

Finally, we move on to Theorem \ref{inequalities_fraser intro}, which gives conditions under which we can control the topology of low index surfaces in some Riemannian $3$-manifolds. 

As usual, we start with a lemma, whose proof can be found in \cite{chen2014}. 

\begin{lemma} \label{trick}
Let $\overline{\Sigma}$ be a closed orientable Riemannian surface of genus $g \geq 0$, and let $h : \overline{\Sigma} \to \mathbb{R}$ be any nonnegative smooth function. Then there exists a conformal map $\overline{f} : \overline{\Sigma} \to \mathbb{S}^2$ such that $\int_{\overline{\Sigma}} \overline{f}h \, \mathrm{d}A = 0$ and $\overline{f}$ has degree less than or equal to $\left[ \frac{g+3}{2} \right]$, where $[x]$ denotes the integer part of $x$.
\end{lemma}

\begin{proof}[Proof of Theorem \ref{inequalities_fraser intro}]
We analyse each one of the three cases.
\begin{itemize}
\item[(i)] Let $\{E_1, E_2, E_3\}$ a local orthonormal frame on $\Sigma$ such that $E_1$ and $E_2$ are tangent to $\Sigma$ and $E_3$ coincides with the unit normal $N$ to $\Sigma$. Since $\Sigma$ is minimal, the Gauss equation gives that
\begin{align*}
K_{\Sigma} = \overline{K}(E_1 \wedge E_2) - \frac{1}{2} \norm{A}^2.
\end{align*}
So,
\begin{align}
\Ric(N) + 2 K_{\Sigma} &= \overline{K}(E_1, N) + \overline{K}(E_2, N) + \overline{K}(E_1,E_2) - \norm{A}^2 \nonumber \\
&= \Ric(E_1) + \Ric(E_2) - \norm{A}^2. \label{ricci equation}
\end{align}
		
Now let $h \geq 0$ be the first eighenfunction of Problem (\ref{problem}). By gluing a disc on each boundary component of $\Sigma$, we may view $\Sigma$ as a compact domain of a closed orientable Riemannian surface $\overline{\Sigma}$. So, Lemma \ref{trick} furnishes a conformal map $\overline{f} : \overline{\Sigma} \to \mathbb{S}^2$ such that $\int_{\overline{\Sigma}} \overline{f} h \, \mathrm{d}A = 0$ and $\deg(\overline{f}) \leq \left[ \frac{g+3}{2} \right]$. Let $f$ denote the restriction of $\overline{f}$ to $\Sigma$. If we write $f = (f_1, f_2, f_3)$, then we know that each component $f_i$ is orthogonal to the first eigenfunction $h$. Since $\Sigma$ has index $1$, we have:
\begin{align*}
Q(f_i,f_i) = \int_\Sigma \left[ \norm{\nabla f_i}^2 - (\Ric(N)+ \norm{A}^2) f_i^2 \right] \, \mathrm{d} A - \int_{\partial \Sigma} q f_i^2 \, \mathrm{d} L \geq 0.
\end{align*}
Summing over $i$ and since $\sum_{i=1}^3 \vert f_i \vert^2 = 1$, we get
\begin{align*}
\int_\Sigma \left[ \norm{\nabla f}^2 - (\Ric(N)+ \norm{A}^2) \right] \, \mathrm{d} A - \int_{\partial \Sigma} q \, \mathrm{d} L \geq 0.
\end{align*}
By the conformality of $\overline{f}$, 
\begin{align} \label{conformality}
\int_{\Sigma} \norm{\nabla f}^2 \, \mathrm{d}A < \int_{\overline{\Sigma}} \norm{\nabla \overline{f}}^2 \, \mathrm{d}A = 2 \Area(\overline{f}(\overline{\Sigma})) = 2 \Area(\mathbb{S}^2) \deg(\overline{f}) \leq 8\pi \left[ \frac{g+3}{2} \right].
\end{align}
Therefore, 
\begin{align*}
\int_{\Sigma} (\Ric(N) + \norm{A}^2) \, \mathrm{d}A + \int_{\partial \Sigma} q \, \mathrm{d}L < 8\pi \left[ \frac{g+3}{2} \right].
\end{align*}
We now use equation (\ref{ricci equation}) and the hypotheses of curvature $\Ric \geq 0$ and $H_{\partial M} \geq 0$ to continue:
\begin{align*}
-2 \int_{\Sigma} K_{\Sigma} \, \mathrm{d}A \leq \int_{\Sigma} (\Ric(E_1) + \Ric(E_2) - 2K_{\Sigma}) \, \mathrm{d}A < 8\pi \left[ \frac{g+3}{2} \right] - \int_{\partial \Sigma} q \, \mathrm{d} L.
\end{align*}
By Gauss-Bonnet theorem,
\begin{align*}
-2\left[ 2\pi(2 - 2g - k) - \int_{\partial \Sigma} k_g \, \mathrm{d}L \right] < 8\pi \left[ \frac{g+3}{2} \right] - \int_{\partial \Sigma} q \, \mathrm{d} L.
\end{align*}
Applying Lemma \ref{mean curvature identity} and rearanging, we arive at
\begin{align*}
\int_{\partial \Sigma} k_g \, \mathrm{d} L < 2\pi \left[ 9 - (-1)^g - 2(g+k) \right],
\end{align*}
as we wanted. The particular case mentioned in the theorem follows immediately.
		
\item[(ii)] To prove this item, we use the identity
\begin{align} \label{gauss identity}
\Ric(N) = \frac{1}{2}(R_M + H_{\Sigma}^2 - \norm{A}^2) - K_{\Sigma}.
\end{align}
Let $f : \Sigma \to \mathbb{S}^2$ be the same function of item (i). As before,
\begin{align*}
\int_\Sigma \left[ \norm{\nabla f}^2 - (\Ric(N)+ \norm{A}^2) \right] \, \mathrm{d} A - \int_{\partial \Sigma} q \, \mathrm{d} L \geq 0.
\end{align*}
Using equation (\ref{gauss identity}) and inequality (\ref{conformality}), we obtain
\begin{align*}
\frac{1}{2}\int_{\Sigma} (R_M + \norm{A}^2) \, \mathrm{d}A - \int_{\Sigma} K_{\Sigma} \, \mathrm{d}A + \int_{\partial \Sigma} q \, \mathrm{d}L < 8\pi \left[ \frac{g+3}{2} \right].
\end{align*}
Now we use Gauss-Bonnet therorem and Lemma \ref{mean curvature identity} to write
\begin{align*}
\frac{1}{2}\int_{\Sigma} (R_M + \norm{A}^2) \, \mathrm{d}A + \frac{1}{\sin \theta} \int_{\partial \Sigma} H_{\partial M} \, \mathrm{d}L < 8\pi \left[ \frac{g+3}{2} \right] + 2\pi(2 - 2g - k).
\end{align*}
Finally, we use the curvature assumptions to get
\begin{align*}
\frac{1}{2} \inf_M R_M \vert \Sigma \vert + \frac{1}{\sin \theta} \inf_{\partial M} H_{\partial M} \vert \partial \Sigma \vert < 2 \pi\left[ 7 - (-1)^g - k \right],
\end{align*}
as we wanted.
		
\item[(iii)] Item (a) follows from Theorem \ref{inf rigidity intro}, since under the curvature hypotheses $\Sigma$ must be a disc. Item (b) follows from from item (ii) of the current theorem (again, $\Sigma$ must be a disc).
\end{itemize}
\end{proof}

\end{document}